\long\def\skipit#1{} 

\newcommand{\mdef}[1]{\textit{#1}}

\newcommand{\ov}{\overline}

\def\dif{{\mathord{{\rm d}}}}

\def\mE{{\mathbb E}}

\def\mN{{\mathbb N}}

\def\mP{{\mathbb P}}

\def\mR{{\mathbb R}}

\def\mU{{\mathbb U}}

\def\mW{{\mathbb W}}

\def\mZ{{\mathbb Z}}

\def\eps{\varepsilon}

\def\cE{{\mathcal E}}

\def\cG{{\mathcal G}}

\documentclass[10pt]{amsart}
\usepackage{pstricks-add}
\usepackage{amsfonts}
\usepackage{amsmath}
\usepackage{amssymb}
\usepackage{ifthen,calc}
\usepackage{color}
\usepackage{epsfig,epstopdf}
\usepackage{graphicx}
\usepackage{latexsym}
\usepackage{multicol}
\usepackage[format=hang, margin=10pt]{caption}
\usepackage[mathscr]{eucal}
\usepackage{bbm}

\newcounter{hours}
\newcounter{minutes}
\newcommand{\printtime}{
	\setcounter{hours}{\time/60}%
	\setcounter{minutes}{\time-\value{hours}*60}
	\ifthenelse{\value{hours}<10}{0}{}\thehours:%
	\ifthenelse{\value{minutes}<10}{0}{}\theminutes}

\numberwithin{equation}{section}
\numberwithin{figure}{section}
\numberwithin{table}{section}

\newtheorem{thm}{Theorem}[section]
\newtheorem{cor}[thm]{Corollary}

\newtheorem{prop}[thm]{Proposition}
\newtheorem{question}[thm]{Question}

\newtheorem{J-com}{JG-comment}[section]
\newtheorem{Rem}[thm]{Remark}

\theoremstyle{definition}
\newtheorem{example}{Example}[section]

\def\dsum{\displaystyle\sum}

\begin{document}

\title{ Limits for  embedding distributions}
\author{Jinlian Zhang}
\address{College of Mathematics and Econometrics, Hunan University, 410082 Changsha, China}
\email{jinlian916@hnu.edu.cn}
\author{Xuhui Peng}
\address{MOE-LCSM, School of mathematics and statistics,  Hunan Normal University, 410081 Changsha, China}
\email{xhpeng@hunnu.edu.cn}
\author{Yichao Chen}
\address{College of Mathematics and Econometrics, Hunan University, 410082 Changsha, China}
\email{ycchen@hnu.edu.cn}
\thanks{Yichao Chen currently works at SuZhou University of Science and Technology}


\begin{abstract}
  In this paper, we first establish a  central limit theorem which is  new in probability, then we find and prove that, under  some  conditions, the embedding distributions of $H$-linear family of graphs with spiders  are asymptotic normal distributions.  As  corollaries, the asymptotic normality for the embedding distributions of path-like sequence of  graphs with spiders and the  genus distributions of ladder-like sequence of graphs are given.  We also prove that the limit of Euler-genus distributions is the same as  that of crosscap-number distributions. The results here can been seen a version of  central limit theorem in topological graph theory.

\end{abstract}

\begin{flushleft}   \vskip-24pt
\end{flushleft}
\bigskip


\maketitle

\section{\large{Introduction}  \label{background} }  


A \textit{graph }is a pair $G = (V, E),$ where $V = V (G)$ is the
 set of vertices, and $E = E(G)$ is  the set of edges. In topological graph theory, a graph {is permitted to have both loops and multiple edges}.  A {\textit{surface}} {$S$}  is a compact connected 2-dimensional manifold
without boundary.  The {\textit{orientable surface}} $O_k(k\geq 0)$ can be obtained from a sphere with $k$ handles attached, where $k$ is called the {\textit{genus}} of $O_k$,  and the {\textit{non-orientable  surface}} $N_j(j\geq 1)$ with $j$ crosscaps, where $j$ is called the {\textit{ crosscap-number}} of $N_j$. The {\textit{Euler-genus}} $\gamma^E$ of a surface $S$ is {given by}
 \begin{eqnarray*}
   \gamma^E=
   \left\{
   \begin{split}
     & 2k,\quad \text{if } S=O_k,
     \\
     & j, \quad   \text{~~if }S=N_j.
   \end{split}
   \right.
 \end{eqnarray*}
 We use $S_i$ to denote the surface {$S$} with Euler-genus $i$, for $i\geq 0$.

{{A graph $G$ is \textit{embeddable }into  a surface $S$} if it can be drawn in the surface such that any edge does not pass through any vertex and any two edges do not cross each other. If $G$ is embedded on the surface $S$, then the components of $S-G$ are the \textit{faces} of the embedding. {A graph embedding is called a  \textit{2-cell {cellular}  embedding} if any simple {closed}  curve in that face can be continuously deformed or contracted in that face to a single point.} {All graph embeddings in the paper are 2-cell {cellular} embeddings}}.

A \mdef{rotation at a vertex} $v$ of a graph $G$ is a cyclic ordering of the edge-ends incident at $v$. A \mdef{{(pure)} rotation system} $\rho$ of a graph $G$ is an assignment of a rotation at every vertex of $G$. A \mdef{general rotation system} for a graph $G$ is a pair $(\rho,\lambda)$, where $\rho$ is a rotation system and $\lambda$ is a map on $E(G)$ with values in $\{0,1\}$. If $\lambda(e)=1$, then the edge $e$ is said to be \mdef{twisted}; otherwise $\lambda(e)=0,$ and we call the edge~$e$ \mdef{untwisted}.  If  $\lambda(e)=0$, for all $e\in E(G)$, then the general rotation system $(\rho,\lambda)$ is  a {pure rotation system}. It is well-known that any graph
embedding can be described by a general rotation system. Let $T$ be a spanning tree of $G,$ a \mdef{$T$-rotation system} $(\rho,\lambda)$ of $G$ is a general rotation system $(\rho,\lambda)$ such that $\lambda(e)=0$, for every edge $e \in E(T)$.  For a fixed spanning tree $T$, two embeddings of $G$ are considered to be \textit{{equivalent}} if their $T$-rotation systems are combinatorially equivalent.  It is known that there is a sequence of vertex-flips that transforms a general rotation system into a $T$-rotation system.

 The number of \textit{(distinct) cellular embeddings} of a graph $G$ on the {surfaces} $O_k,N_j,$ and $S_i$
 are denoted by $\gamma_k(G),\tilde{\gamma}_j(G)$, and $\eps_i(G),$ respectively.
{By the \emph{genus distribution} of a graph $G$ we mean the sequence}
$$\gamma_0(G),\gamma_1(G),\gamma_2(G),\cdots ,$$
and  the \emph{genus polynomial} of $G$ is
$$\Gamma_G(x)=\sum_{k=0}^\infty \gamma_k(G)x^k.$$

\noindent Similarly, we have \textit{the crosscap-number distribution} $\{\tilde{\gamma}_i(G)\}_{n=1}^{\infty}$ and \textit{the Euler-genus distribution} $\{\eps_i(G)\}_{n=1}^{\infty}.$
{The \emph{crosscap-number polynomial} $ \tilde{\Gamma}_G(x)$
and the \emph{Euler-genus polynomial} $ \cE_G(x)$ of $G$ are defined analogously.}

For a deeper discussion of the above concepts, we may refer the reader to \cite{CJ94,CG18}.
\noindent {The following assumption will be needed throughout the paper. When we say \textit{embedding distribution} of a graph $G$, we mean its genus distribution, crosscap-number distribution or Euler-genus distribution.}

 Usually,  the embedding distribution of a graph $G$ with  tractable size can been calculated explicitly,  we  still concern  the  global feature of the embedding  distribution of  $G$.
For example: \textbf{(1)}   {\textit{Log-concavity}}. For this aspect, we refer to \cite{GMT,GMTW,GMTW16-3,GMTW16-2,SS97} etc.
\textbf{(2) } {\textit{Average genus},  \textit{average crosscap-number} and \textit{average  Euler-genus}.}
The {\textit{average genus}} of graph $G $ is given by \begin{eqnarray*}
    {\gamma_{avg}(G)}&=&\frac{\Gamma_{G}'(1)}{\Gamma_{G}(1)}= \sum_{k=0}^{\infty } \frac{k \cdot \gamma_k(G)}{\Gamma_{G}(1)},
    \end{eqnarray*}
 the {\textit{average crosscap-number}} $\tilde{\gamma}_{avg}(G)$ and {\textit{average Euler-genus}} $\eps_{avg}(G)$ of a graph $G $ is  similarly defined.
The study of average genus,  average crosscap-number and average  Euler-genus {received many attentions in
topological graph theory}. For researches on this aspect, one can see \cite{CJ92,SS83,White} etc. There is also a notation of \emph{variance}.  For example,  see that in   \cite{SS90}.
The {\textit{variance}} of the genus distribution
of the graph  $G$ is given by
\begin{align*}
\gamma_{var}(G)=\sum_{k=0}^{\infty }\big(k-\gamma_{avg}(G)\big)^2\cdot \frac{\gamma_k(G)}{\Gamma_{G}(1)}.
\end{align*}
 We define  the  variance of crosscap-number distribution $\tilde{\gamma}_{var}(G)$ and of Euler-genus distribution $\eps_{var}(G)$  similarly.


The motivation of this article is as follows.
Let  $\{ G_n^\circ\}_{n=1}^\infty$
be a sequence of linear family of graphs  with spiders whose definition is given in subsection 3.1,
and we denote the  embedding distribution of graph $G_n^\circ$  by $\{p_i(n)\}_{i=0}^{\infty } .$
The normalized sequence of $\{p_i(n)\}_{i=0}^{\infty } $ is
\begin{eqnarray*}
  \frac{p_i(n)}{\sum_{k=0}^\infty  p_k(n)},\quad i=0,1,\cdots.
\end{eqnarray*}
Then, the above sequence is a distribution    in probability, we denote it by $F_n$.
One problem appears, when $n$ is big enough, whether the  distribution    $F_n$  will look like some well-known distribution  in probability. If the answer is yes, then it demonstrates   the outline of   embedding  distribution for   graph $G_n^\circ$ when $n$ is big enough.
In the point of mathematics, this is to seek the limit for $F_n$ or  the embedding  distribution  of  graph   $G_n^\circ$.

In this paper, we make  researches on
 the embedding distributions which are   closely related to the above problem. Under some  weak conditions,
 we show  the embedding  distributions (genus,  crosscap-number  or   Euler-genus distributions)
  of $G_n^\circ$ are {\textit{asymptotically normal distribution} }when $n$ tends to infinity.
 We say  the embedding  distributions (genus,  crosscap-number  or   Euler-genus distributions)
  of $G_n^\circ$ are \textit{asymptotically normal distribution} with mean $\mu_n$ and variance $\sigma_n^2$ if \begin{eqnarray*}
    \lim_{n\rightarrow \infty}\sup_x\left|\sum_{i\leq \sigma_nx+\mu_n}p_i(n)-\frac{1}{\sqrt{2\pi}}\int_{-\infty}^xe^{-t^2/2}\dif t\right|=0.
  \end{eqnarray*}
 Since normal distributions have many very good properties,   the genus distributions (crosscap-number  or   Euler-genus distributions)
  of $G_n^\circ$ also have many good properties when $n$ is big enough. Such as:
(1) {\textit{Symmetry}}, normal distributions are symmetric around their mean. (2)
{\textit{Normal distributions are defined by two parameters, the mean $\mu$ and the standard deviation $\sigma$.
}}Approximately 95\% of the area of a normal distribution is within two standard deviations of the mean.
This implies that the genus distributions
 of $G_n^\circ$ are mainly concentrated
on the interval $(\gamma_{avg}(G_n^\circ)-2\sqrt{\gamma_{var}(G_n^\circ)},\gamma_{avg}(G_n^\circ)+2\sqrt{\gamma_{var}(G_n^\circ)}]$ when $n$ {is} big enough.
Similar results also hold for
crosscap-number  and   Euler-genus distributions.
We also show that  the genus distributions (crosscap-number  or   Euler-genus distributions)
  of $G_n^\circ$ are not always    asymptotically normal distribution.
 This  is  the {first} time someone  prove  the   embedding   distributions   of some families  of graphs   are   asymptotically normal.

 In Section 2,  we  establish a   central limit theorem which is also new in probability.
In Section 3, we apply this central limit theorem to the embedding distributions of $G_n^\circ$ and give their  limits.
 In Section 4, some examples are demonstrated.

\section{A central limit theorem }
For a non-negative integer  sequence$\{p_i(n)\}_{i=0}^\infty$, let  $P_n(x) =\sum\limits_{i=0}^{\infty}p_{i}(n)x^i,x\in\mR$.
In this section, we always assume  $P_n(x)$
 satisfies  a   $k^{th}$-order homogeneous linear recurrence relation
\begin{eqnarray}
\label{wp-12}
  P_{n}(x)=b_1(x)P_{n-1}(x)+b_2(x)P_{n-2}(x)+\cdots+
  b_k(x)P_{n-k}(x),
\end{eqnarray}
where  $b_j(x)$$(1\leq j \leq k)$ are polynomials with  integer coefficients.
We define a   polynomial associated with  (\ref{wp-12})
 \begin{eqnarray}
\label{555-1}
F(x,\lambda)=\lambda^k-b_1(x)\lambda^{k-1}-b_2(x)\lambda^{k-2}\cdots-b_{k-1}(x)\lambda-b_k(x).
\end{eqnarray}

Obviously,  for any $x\in \mR,$ there exist  some $r=r(x)\in\mN$,  $m_1(x),\cdots, m_r(x)\in \mN$  with $\sum_{i=1}^r m_i(x)=k$,  and  numbers $\lambda_i(x),i=1,\cdots,r$ with    $|\lambda_1(x)|\geq |\lambda_2(x)|\geq |\lambda_3(x)|\geq \cdots |\lambda_r(x)|$  such that
\begin{eqnarray}
\label{4-1}
 F(x,\lambda)=
 (\lambda-\lambda_1(x){)}^{m_1(x)}\cdots (\lambda-\lambda_r(x)^{m_r(x)}.
\end{eqnarray}
 And the general solution to  (\ref{wp-12}) is given by
\begin{eqnarray}
\label{pz-1}
\begin{split}
P_n(x) = \sum_{i=1}^r \lambda^n_i(x)\big(a_{i,0}(x)+a_{i,1}(x)n+\cdots+ a_{i,m_i(x)-1}(x)n^{m_i(x)-1}\big).
 \end{split}
\end{eqnarray}

Let
\begin{eqnarray}
\label{ppp-1}
  e=\frac{\lambda_1'(1)}{D},\quad  v=\frac{-\big(\lambda_1'(1)\big)^2 +D\cdot  \lambda_1^{''}(1)
 +D\cdot\lambda_1'(1) }{D^2}.
\end{eqnarray}
where $D=\lambda_1(1)$.
For any $n\in \mN,$  let  $X_n$ be  a random variable with distribution
\begin{eqnarray*}
\mP(X_n=i)=\frac{p_{i}(n)}{P_n(1)}, \quad i=0,1,\cdots,
\end{eqnarray*}

{The remainder of  this section is  devoted to the proof of the following theorem.}
\begin{thm}
\label{11-3}
Let $P_n(x) =\sum_{i=0}^{\infty}p_{i}(n)x^i, n\geq k+1$ be   polynomials satisfying   (\ref{wp-12}).
At  $x=1,$ suppose the multiplicity of maximal root for  polynomial (\ref{555-1})  is 1.  Then  the following results hold  depending on the value of $v.$
\begin{description}
  \item[Case I] $ v>0$.
  The law  of  $X_n $  is asymptotically normal  with mean    $e\cdot n$ and variance $ v\cdot n$ when $n$ tends to infinity. That is,
  \begin{eqnarray*}
  \lim_{n\rightarrow \infty } \sup_{x\in \mR }\left|\mP(\frac{X_n-e\cdot n}{ \sqrt{v \cdot n}} \leq x)-\int_{-\infty}^x \frac{1}{\sqrt{2\pi}}e^{-\frac{1}{2}u^2}\dif u\right|=0.
  \end{eqnarray*}
  In particular,    we have
    \begin{eqnarray*}
   \lim_{n\rightarrow \infty } \sup_{x\in \mR}\left|\frac{1}{P_n(1)}\sum_{\tiny{0\leq i\leq x \sqrt{ v\cdot n}+e\cdot n]}}p_i(n)-\int_{-\infty}^x \frac{1}{\sqrt{2\pi}}e^{-\frac{1}{2}u^2}\dif u\right|=0.
  \end{eqnarray*}
  \item[Case II] $v<0$. This case is impossible to appear.
  \item[Case III] $ v=0$.
  For any $\alpha>\frac{1}{3}$,
  the  law   of  $\frac{X_n-e\cdot n}{n^\alpha} $  is asymptotically one-point distribution concentrated  at $0$. In more accurate words, the following holds.
   \begin{eqnarray*}
  \label{12-4}
    \lim_{n\rightarrow \infty } \mP\left(\frac{X_n-e\cdot n}{n^\alpha} \leq x\right)=
    \Bigg\{\begin{split}
      &1, \quad \text{ if  $x\geq 0,$ }
      \\
      &0, \quad \text{ else. }
    \end{split}
  \end{eqnarray*}
  Furthermore, if  all these  functions   $b_1(x),\cdots,b_k(x)$ are   constant,
  then the limits of the law of $X_n$ is  a  discrete distribution.  That is,  for some  $\kappa\in \mN$
  and $\omega_j,j=0,\cdots,\kappa$ with $\sum\limits_{j=1}^\kappa\omega_j=1$,  we have
  \begin{eqnarray}
  \label{1212-1}
   \lim_{n\rightarrow \infty}\mP(X_n=j)=\omega_j,j=0,\cdots,\kappa.
  \end{eqnarray}
  \end{description}
\end{thm}

\begin{proof}
One arrives at that
\begin{eqnarray}
\label{q-5}
D=\lambda_1(1)> |\lambda_2(1)|\geq |\lambda_3(1)|\geq \cdots |\lambda_r(1)|
\text{ and }  m_1(1)=1.
\end{eqnarray}
By  \cite{NP94} and the smooth of $F$, for $i=1,\cdots,r $,  $\lambda_i(x) $ are continuous.
Thus, for some  $\delta>0$, we have
\begin{eqnarray}\label{wp-11}
\lambda_1(x)> |\lambda_2(x)|\geq |\lambda_3(x)|\geq \cdots |\lambda_r(x)|, \quad\forall x\in (1-\delta,1+\delta).
\end{eqnarray}
For $m_1(x)$ and $\lambda_1(x),$ we have the following fact.

Fact: for some $\delta>0$,    we have $\lambda_1(x)$ is smooth on $(1-\delta,1+\delta)$
and
\begin{eqnarray}
\label{ww-2}
m_1(x)=1, \quad \forall x\in (1-\delta,1+\delta).
\end{eqnarray}
One easily sees that
\begin{eqnarray*}
 F(1,\lambda) = (\lambda-\lambda_1(1))(\lambda-\lambda_2(1))^{m_2(1)}\cdots(\lambda-\lambda_r(1))^{m_r(1)}.
\end{eqnarray*}
and
\begin{eqnarray}
\label{ww-1}
  \frac{\partial F(x,\lambda)}{\partial \lambda}\Big|_{x=1,\lambda=\lambda_1(1)}\neq 0.
\end{eqnarray}
Actually, $\lambda_1(x)$ can be seen an implicit function decided by
 \begin{eqnarray*}
  F(x,\lambda)=0.
 \end{eqnarray*}
 By the  smooth of $F$ and   the implicit function theorem,
     $\lambda_1(x)$ is   smooth on $(1-\eps,1+\eps)$ for some $\eps>0$.
By the smooth  of $F$ and (\ref{ww-1}), for some $\eps>0,$ we have
     \begin{eqnarray*}
   \frac{\partial F(x,\lambda)}{\partial \lambda}\neq 0, \quad \forall x\in (1-\eps,1+\eps),\forall \lambda\in (\lambda_1(1)-\eps,\lambda_1(1)+\eps)
  \end{eqnarray*}
  which yields the desired result (\ref{ww-2}).

Combining (\ref{q-5}),   the general solution to  (\ref{wp-12}) is given by
\begin{eqnarray}
\label{pz-1}
\begin{split}
P_n(x)& = a(x)\lambda_1^n(x)
 \\ &  \quad +\sum_{i=2}^r \lambda^n_i(x)\big(a_{i,0}(x)+a_{i,1}(x)n+\cdots+ a_{i,m_i(x)-1}(x)n^{m_i(x)-1}\big).
 \end{split}
\end{eqnarray}

We consider the following three  different cases.

\textbf{Case}  \textbf{\uppercase\expandafter{\romannumeral1}}: $ v>0$.
 Let  $Y_n=\frac{X_n-en}{\sqrt{vn}}$ and $\phi_{Y_n}(t)=\mE e^{\mathbbm{i}tY_n}$ be the characteristic function of $Y_n$, where $\mathbbm{i}$ is a complex number with $\mathbbm{i}^2=-1$.

 In order to  prove
 \begin{eqnarray*}
  \lim_{n\rightarrow \infty}  \sup_{x\in \mR}\left|\mP(Y_n\leq x)-\int_{-\infty}^x \frac{1}{\sqrt{2\pi}}e^{-\frac{u^2}{2}}\dif u\right|=0,
 \end{eqnarray*}
by  \cite[Chapter 1]{PVV} and  the  continuity theory(\cite[Chapter 15]{Feller}) for characteristic function   in probability,   we only need to prove
  \begin{eqnarray}
  \label{9-1}
 && \lim_{n\rightarrow \infty}\phi_{Y_n}(t)=\lim_{n\rightarrow \infty}\mE e^{\mathbbm{i}\cdot t\frac{X_n-en}{\sqrt{vn}}}=\int_{\mR}e^{\mathbbm{i}\cdot tu}\frac{1}{\sqrt{2\pi}}e^{-\frac{u^2}{2}}\dif u=e^{-\frac{t^2}{2}},\quad \forall t.
\end{eqnarray}
We will give a proof of this.

Let $
  a_n=\frac{1}{\sqrt{v n}},b_n=\sqrt{\frac{e^2n}{v}}
$ and $y=e^{a_nt}$.  {By these definitions, one easily sees that
\begin{eqnarray}
\label{c-1}
\begin{split}
  & \frac{n\lambda_1'(1)}{D}a_n-b_n=ne\frac{1}{\sqrt{v n}}-\sqrt{\frac{e^2n}{v}}=0,
  \\
  & n\frac{\lambda_1'(1)D+\lambda_1''(1)D-\lambda'(1)^2}{2D^2}a_n^2
  =n\cdot \frac{v}{2}\cdot \frac{1}{vn}=\frac{1}{2}.
  \end{split}
\end{eqnarray}}

{By Taylor formula, we have}

 \begin{eqnarray*}
\ln \frac{\lambda_1(y)}{D} &=& \frac{\lambda_1'(1)}{D}(y-1)
+\frac{1}{2D^2}\big[\lambda_1^{''}(1)D-(\lambda_1'(1))^2\big](y-1)^2
+o((y-1)^2),
\end{eqnarray*}
and
\begin{eqnarray*}
y&=&1+a_nt+\frac{a_n^2t^2}{2}+o(a_n^2t^2).
\end{eqnarray*}
Therefore, by $\lim\limits_{n\rightarrow \infty}y=1,$ it holds that
\begin{eqnarray*}
\ln \frac{\lambda_1(y)}{D} &=&\frac{\lambda_1'(1)}{D}\big(a_nt+\frac{1}{2}a_n^2t^2+o(a_n^2t^2)\big)
\\ &&  +\frac{1}{2D^2}\big[\lambda_1^{''}(1)D-(\lambda_1'(1))^2\big]\big(a_nt+\frac{1}{2}a_n^2t^2+o(a_n^2t^2)\big)^2
\\ && +o\Big(\big(a_nt+\frac{1}{2}a_n^2t^2+o(a_n^2t^2)\big)^2\Big).
\\
&=& \frac{\lambda_1'(1)}{D}\big(a_nt+\frac{1}{2}a_n^2t^2\big)
+\frac{1}{2D^2}\big[\lambda_1^{''}(1)D-(\lambda_1'(1))^2\big]a_n^2t^2
+o(\frac{1}{n}),
\end{eqnarray*}
where in the last  equality, we have used  $a_n^2t^2=\frac{t^2}{vn}.$
 By   (\ref{q-5}) and the above equality, we obtain
\begin{eqnarray*}
 && \lim_{n\rightarrow \infty}
 \mE e^{t\frac{X_n-en}{\sqrt{vn}}}
 = \lim_{n\rightarrow \infty}\mE e^{a_ntX_n-b_nt}
 =\lim_{n\rightarrow \infty} \frac{P_n(e^{a_nt})e^{-b_nt}}{P_n(1) }
 = \lim_{n\rightarrow \infty} \frac{P_n(y)e^{-b_nt}}{a(1)  D^n }
 \\&& =\lim_{n\rightarrow \infty}   \frac{\lambda^n_1(x) a(y)}{ a(1) D^n }e^{-b_nt}=\lim_{n\rightarrow \infty} \frac{\lambda^n_1(y) }{D^n }e^{-b_nt}
  =\lim_{n\rightarrow \infty}\exp\{n \ln \frac{\lambda_1(y)}{D}-b_nt\}
   \\ && =\!\lim_{n\rightarrow\infty}\! \exp\!\!\left\{\! \frac{n\lambda_1'(1)}{D}\big(a_nt+\frac{1}{2}a_n^2t^2\big)
\!+\!\frac{n}{2D^2}\big[\!\lambda_1^{''}(1)D-(\lambda_1'(1))^2\big]a_n^2t^2
\!-b_nt+n \! \cdot \!o(\frac{1}{n})\!\right\}
  \\ && =e^{\frac{t^2}{2}},
\end{eqnarray*}
where in the last equality, we have used (\ref{ppp-1}).

In the above equality, we replace $t$ by $\mathbbm{i}t$. Then, we   get
\begin{eqnarray*}
 && \lim_{n\rightarrow \infty}\mE e^{\mathbbm{i}\cdot t\frac{X_n-en}{\sqrt{vn}}}=e^{-\frac{t^2}{2}}
\end{eqnarray*}
which finishes the proof of (\ref{9-1}).

\textbf{
Case}  \textbf{\uppercase\expandafter{\romannumeral2}}: $ v<0$.
Let  $Y_n=\frac{X_n-en}{\sqrt{-vn}}$ and $\phi_{Y_n}(t)=\mE e^{\mathbbm{i}tY_n}$.
Set $
  a_n=\frac{1}{\sqrt{-v n}},b_n=\sqrt{\frac{e^2n}{-v}}
$ and $y=e^{a_nt}$.

{  By Taylor formula, we have
\begin{eqnarray*}
\ln \frac{\lambda_1(y)}{D} &=& \frac{\lambda_1'(1)}{D}\big(a_nt+\frac{1}{2}a_n^2t^2\big)
+\frac{1}{2D^2}\big[\lambda_1^{''}(1)D-(\lambda_1'(1))^2\big]a_n^2t^2
+o(\frac{1}{n}).
\end{eqnarray*}}
{Then, by (\ref{q-5}), we have
\begin{eqnarray*}
 && \lim_{n\rightarrow \infty}
 \mE e^{t\frac{X_n-en}{\sqrt{-vn}}}
 = \lim_{n\rightarrow \infty}\mE e^{a_ntX_n-b_nt}
 =\lim_{n\rightarrow \infty} \frac{P_n(e^{a_nt})e^{-b_nt}}{a(1) D^n }
 = \lim_{n\rightarrow \infty} \frac{P_n(y)e^{-b_nt}}{a(1) D^n }
 \\&& =\lim_{n\rightarrow \infty}   \frac{\lambda^n_1(y) a(y)}{ a(1) D^n }e^{-b_nt}=\lim_{n\rightarrow \infty} \frac{\lambda^n_1(y) }{D^n }e^{-b_nt}
=\lim_{n\rightarrow \infty}\exp\{n \ln \frac{\lambda_1(y)}{D}-b_nt\}
\\  && =\!\lim_{n\rightarrow\infty}\! \exp\!\!\left\{\! \frac{n\lambda_1'(1)}{D}\big(a_nt+\frac{1}{2}a_n^2t^2\big)
\!+\!\frac{n}{2D^2}\big[\!\lambda_1^{''}(1)D-(\lambda_1'(1))^2\big]a_n^2t^2
\!-b_nt+n \! \cdot \!o(\frac{1}{n})\!\right\}
\\ && =e^{-\frac{t^2}{2}}.
\end{eqnarray*}}
In the above equality, we replace $t$ by $\mathbbm{i}t$ and get
\begin{eqnarray}
\label{www-2}
 && \lim_{n\rightarrow \infty}\mE e^{\mathbbm{i}\cdot t\frac{X_n-en}{\sqrt{vn}}}=e^{\frac{t^2}{2}}.
\end{eqnarray}
On the other hand, by the properties of characteristic function, one sees
\begin{eqnarray*}
\big|\mE e^{\mathbbm{i}\cdot t\frac{X_n-en}{\sqrt{vn}}}\big|\leq 1
\end{eqnarray*}
which contradicts with (\ref{www-2}) and gives the desired result.

\textbf{Case} \textbf{\uppercase\expandafter{\romannumeral3}}: $ v=0$.
Let  $Y_n=\frac{X_n-en}{n^{\alpha}},a_n=n^{-\alpha},b_n=en^{1-\alpha}$ and  $y=e^{a_n t}$.

By Taylor formula and $v=0$, we have
\begin{eqnarray*}
\ln \frac{\lambda_1(y)}{D} &=& \frac{\lambda_1'(1)}{D}(y-1)
+\frac{1}{2D^2}\big[\lambda_1^{''}(1)D-(\lambda_1'(1))^2\big](y-1)^2
+O((y-1)^3)
\\ &=& \frac{\lambda_1'(1)}{D}(a_n t+\frac{1}{2}a_n^2t^2)
+\frac{1}{2D^2}\big[\lambda_1^{''}(1)D-(\lambda_1'(1))^2\big]a_n^2t^2
+O(n^{-3\alpha})
\\ &=& \frac{\lambda_1'(1)}{D}a_nt
+O(n^{-3\alpha})= e n^{-\alpha}t
+O(n^{-3\alpha}).
\end{eqnarray*}
Therefore,
\begin{eqnarray*}
 \lim_{n\rightarrow \infty} \exp\{n \ln \frac{\lambda_1(y)}{D}-b_n t\}&=&\lim_{n\rightarrow \infty} \exp\{{n\cdot  e n^{-\alpha}t+n\cdot O(n^{-3\alpha})-en^{1-\alpha}t}\}=1.
\end{eqnarray*}
{Then,  by  (\ref{q-5}), we have}
{\begin{eqnarray*}
 && \lim_{n\rightarrow \infty}
 \mE e^{t\frac{X_n-en}{n^\alpha}}
 = \lim_{n\rightarrow \infty}\mE e^{ta_n X_n-b_nt}
 =\lim_{n\rightarrow \infty} \frac{P_n(y)e^{-b_nt}}{a(1) D^n }
 = \lim_{n\rightarrow \infty} \frac{P_n(y)e^{-b_nt}}{a(1) D^n }
 \\&& =\lim_{n\rightarrow \infty}   \frac{\lambda^n_1(y) a(y)}{ a(1) D^n }e^{-b_nt}=\lim_{n\rightarrow \infty} \frac{\lambda^n_1(y) }{D^n }e^{-b_nt}
  =\!\lim_{n\rightarrow \infty}\exp\{n \ln \frac{\lambda_1(y)}{D}-b_nt\}
\!=1.
\end{eqnarray*}}
In the above equality, we replace $t$ by $\mathbbm{i}t$ and get
\begin{eqnarray*}
 && \lim_{n\rightarrow \infty}\mE e^{\mathbbm{i}\cdot t\frac{X_n-en}{n^\alpha}}=1
\end{eqnarray*}
which  yields  the desired result.

Now we give a proof of   (\ref{1212-1}).
Noting  (\ref{pz-1}), we have
\begin{eqnarray*}
 P_n(x)=a(x)\lambda_1^n+\sum_{i=2}^r (a_{i,0}(x)+a_{i,1}(x)n+\cdots +a_{i,m_i(x)-1}(x)n^{m_i(x)-1})\lambda_i^n,
\end{eqnarray*}
where $m_i(x)=m_i,i=2,\cdots,r$ are constants.
Since $P_n(x)$  is a polynomial of $x$ and  $\lambda_i,i=1,\cdots,r$ are constant, $a(x)$ is also a polynomial of $x$. Assume
\begin{eqnarray*}
  a(x)=\sum_{j=0}^\kappa c_jx^j.
\end{eqnarray*}
then,
\begin{eqnarray*}
  \lim_{n\rightarrow \infty}\mE e^{tX_n}
  =\lim_{n\rightarrow \infty}\frac{P_n(e^t)}{P_n(1)}= \lim_{n\rightarrow \infty}\frac{a(e^t)D^n }{a(1) D^n}=
  \frac{a(e^t)}{a(1)}=\sum_{j=0}^\kappa \omega_je^{jt},
\end{eqnarray*}
where $\omega_j=\frac{c_j}{a(1)}.$
In the above equality, we replace $t$ by $\mathbbm{i}t$ and get
\begin{eqnarray}
\label{www-1}
 && \lim_{n\rightarrow \infty}\mE e^{\mathbbm{i}\cdot t X_n}=\sum_{j=0}^\kappa \omega_je^{\mathbbm{i}jt}.
\end{eqnarray}
By  the  continuity theory(\cite[Chapter 15]{Feller}) for characteristic function   in probability,  we obtain   the desired result.
\end{proof}

By (\ref{1212-1}), the condition $v>0$ is  necessary to ensure  the asymptotic normality.
  In the end of this section, we give a remark here to explain the $e,v$ appeared in Theorem  \ref{11-3}.
\begin{Rem}
In  the special case, when  $
P_n(x) = a(x)\lambda_1^n(x)
$,  it holds that
$$\mE X_n= \frac{P_n'(1)}{P_n(1)} =e\cdot n+O(1)
$$
and
$$
Var(X_n)=\frac{P_n''(1)+P_n'(1)}{P_n(1)}-\big(\frac{P_n'(1)}{P_n(1)}\big)^2=v\cdot n+O(1).$$
Since $\lambda_1(x)>\lambda_i(x),i=2,\cdots,k,$  we can expect that for the $P_n(x)$ in
(\ref{pz-1}) and   $
P_n(x) = a(x)\lambda_1^n(x),$ they have the same  asymptotic
mean  and variance when $n$ tends to infinity.
\end{Rem}

\section{The limits for embedding distributions }

In this section, we consider the  limit for embedding distributions of \emph{$H$-linear family  of graphs  with spiders} {$\{G_n^\circ\}_{n=1}^\infty$}. In subsection 3.1, we give a definition of $G_n^\circ$. In subsection 3.2, we briefly describe the production matrix.
Then we give  the limit of   embedding distribution for graph $G_n^\circ$ in subsection 3.3.
Finally, we demonstrate the relation between the limit of crosscap-number distributions and  Euler-genus distributions  in subsection 3.4.

\subsection{$H$-linear family of graphs with spiders}The definition of
\emph{$H$-linear family of graphs  with spiders}, gave by
Chen and Gross \cite{CG19}, is a generalization of  $H$-linear family of graphs introduced by Stahl \cite{SS91}.
Suppose $H$ is a connected graph.Let $U=\{u_1,\cdots,u_s\}$ and
$V=\{v_1,\cdots,v_s\}$ be {two} disjoint subsets of $V(H)$. For $i=1,2,\cdots,$ {suppose $H_i$ is a copy of $H$ }and let $f_i:H\rightarrow H_i$ be an isomorphism. For each $i\geq 1$ and $1\leq j\leq s,$ we set $u_{i,j}=f_i(u_j)$ and $v_{i,j}=f_i(v_j).$ An \emph{$H$-linear family of graphs}, {denoted by ${\cG}=\{{G}_n\}_{n=1}^\infty$, is defined as follows}:
\begin{itemize}
\item ${G}_1=H_1.$
\item ${G}_n$ is constructed by ${G}_{n-1}$ and $H_n$ be amalgamating the vertex $v_{n-1,j}$ of ${G}_{n-1}$ with the vertex $u_{n,j}$ of $H_n$ for $j=1,\cdots,s.$
\end{itemize}
\begin{figure}[h]
  \centering
  \includegraphics[width=0.55\textwidth]{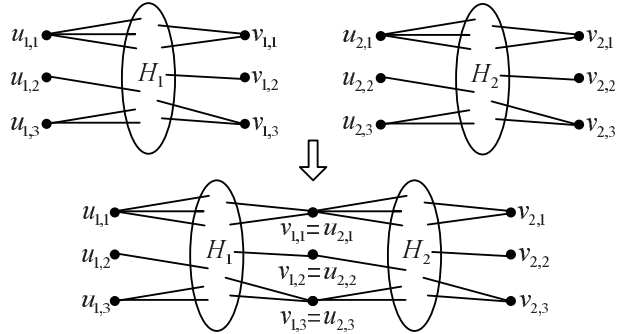}
  \caption{The graph ${G}_{2}$ in a generic $H$-linear sequence}
  \label{xianxing}
\end{figure}
Figure \ref{xianxing} shows an example for the graphs $H_1,$ $H_2$  and $G_2.$

Now, we inroduce the definition of $H$-linear family of graphs with spiders.
For $1\leq j\leq s,$ {let $(J_j,{t_{j,i'}})$ }and $(\overline{J_j},{\overline{t_{j,i'}}})$ be graphs in which $\{t_{j,i'}\},\{\overline{t_{j,i'}}\}$, respectively, are sets of root-vertices.
{For $1\leq i\leq s,$   $I_i$ and $\bar{I}_i$ are  subsets of $\{1,\cdots,s\}.$}
 A graph $\{{G}_n^\circ\}_{n=1}^\infty$is constructed  from ${G}_n$ by  amalgamating the vertex $u_{1,i}$ of $G_n$ with the vertex $t_{j,i'}$ of $J_j$ { for $j\in I_i$}, and amalgamating the vertex $v_{n,i}$ of ${G}_n$ with $\overline{t_{j,i'}}$ of $\overline{J_j}$ { for $j\in \bar{I}_i.$ }  The graphs  $(J_j,{t_{j,i'}})$ and $(\overline{J_j},{\overline{t_{j,i'}}})$ are called \emph{spiders }for the sequence ${\cG}.$ The resulting sequence of graphs is said to be an \emph{$H$-linear family of graphs with spiders} and is denoted by ${\cG}^o$.
We call  $\cG^{\circ}$ \mdef{ring-like}, if there  is a spider among $J_1, J_2, \ldots, J_s$ that coincide with a spider among $\ov J_1, \ov J_2, \ldots, \ov J_s.$
{ The graphs in Figure \ref{lizi} demonstrate an example of
\mdef{ring-like}, in which $s=1$ and $\bar{J}_1=J_1.$ }

\begin{figure}[h]
  \centering
  \includegraphics[width=0.70\textwidth]{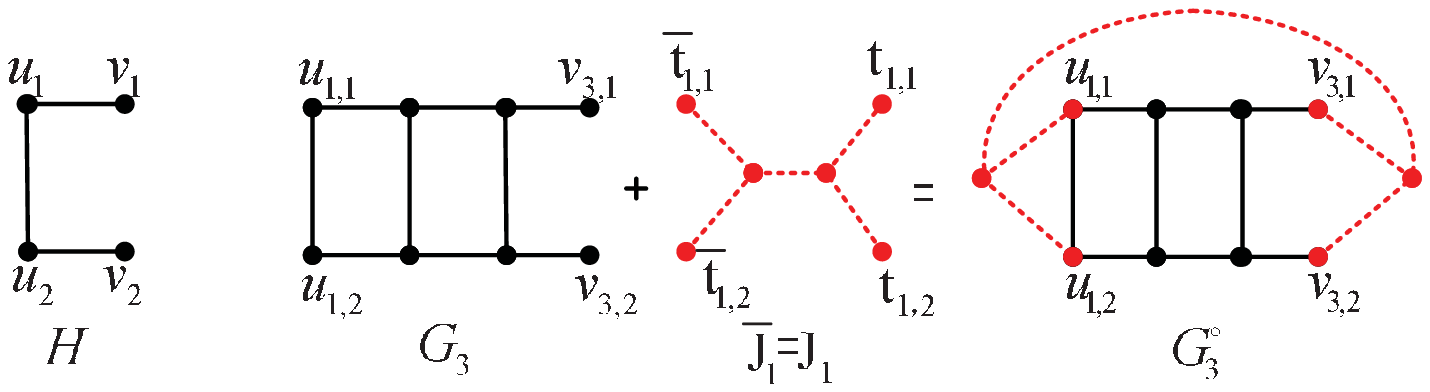}
  \caption{Using a spider to construct  ${G}^\circ_3$ }
  \label{lizi}
\end{figure}

\subsection{Production matrix} By permutation-partition pairs, Stahl \cite{SS91} showed that the calculation of genus polynomial of $G_n$ can be converted to  a (transfer) matrix  method.
Such matrices are also called {production matrices} \cite{GKMT18} or transfer matrix \cite{Moh15} (using different techniques and methods).
 Here we briefly describe the production matrix of { $\cG$} (or { $\cG^\circ $}).  For more details on this,  see \cite{Moh15,GKMT18} etc. We refer to \cite{CJ94} for face-tracing algorithm.


We suppose that there are $k$ embedding types for the graph $G_n$ with roots $u_{1,1}, u_{1,2},\ldots,u_{1,s},$ $v_{n,1},v_{n,2},$
$\ldots,v_{n,s}$. For $1\leq j\leq k,$ let $\gamma_k^j(G_n)$ be the number of embeddings of $G_n$ in $O_k$ with type $j$
and
\begin{eqnarray*}
  \Gamma_{G_n}^j(x)=\sum_{i\geq 0}\gamma_i^j(G_n)x^i.
\end{eqnarray*}
From the definition $H$-linear family of graphs and face-tracing algorithm, we obtain
\begin{eqnarray}
\label{q-2}
  ( \Gamma_{G_n}^1(x), \Gamma_{G_n}^2(x),\cdots, \Gamma_{G_n}^k(x))^T=M(x)\cdot ( \Gamma_{G_{n-1}}^1(x), \Gamma_{G_{n-1}}^2(x),\cdots, \Gamma_{G_{n-1}}^k(x))^T.\notag\\
\end{eqnarray}
where { $\alpha^T$  is the transpose of the  vector $\alpha$ }and
 \begin{eqnarray*}
   M(x)=  \left[\begin{matrix}
   m_{1,1}(x) & m_{1,2}(x) &\cdots &m_{1,k}(x)
     \\    m_{2,1}(x) & m_{2,2}(x) &\cdots &m_{2,k}(x)
     \\    \vdots  & \vdots &\ddots&\vdots
     \\  m_{k,1}(x) & m_{k,2}(x) &\cdots &m_{k,k}(x)
   \end{matrix}
   \right]
 \end{eqnarray*}
 is  the \emph{transfer matrix} \cite{Moh15} or \textit{\textit{production matrix} \cite{GKMT18} of {genus distribution} }of { $\cG$ (or $\cG^{\circ}$}). In \cite{GKMT18}, the authors showed that  the production matrix  $M(x)$ can been calculated with a computer program.

Let
\begin{eqnarray*}
  V_{G_n}(x)=( \Gamma_{G_n}^1(x), \Gamma_{G_n}^2(x),\cdots, \Gamma_{G_n}^k(x))^T,
\end{eqnarray*}
another property for the genus polynomial of  $H$-linear sequence with spiders is that there exists a $k$-dimensional row-vector $$
V=(v_1(x),\cdots,v_k(x))
$$
such that {$\Gamma_{G_n^\circ}(x)=V\cdot V_{G_n}(x). $ Note  that if there are no spiders, that is $\cG^\circ=\cG$, then $V=(1,1,\cdots,1).$

For the case of  Euler-genus polynomials,
 there also has the  \textit{production matrix } \textit{of Euler-genus distribution} of  $\cG$ (or $\cG^{\circ}$). We take  \cite{CG18} as an   example of this.

\begin{example}Suppose $P_n$ is the path graph on $n$ vertices. An \textit{ladder graph} $L_n$ is obtained by taking the graphical cartesian product of the path graph $P_{n}$ with $P_2$, i.e. $L_n=P_{n}\Box P_2.$ From Section 3 in \cite{CG18}, the $2\times 2$ production matrix of Euler-genus distribution of the ladder graph is  given by
\begin{eqnarray*}
 M(x)=\left[
                \begin{array}{cc}
                  2 & 4 \\
                  2x+4x^2 & 4x \\
                \end{array}
              \right].
\end{eqnarray*}
It follows that
\begin{eqnarray*}
M(1)=\left[
                \begin{array}{cc}
                  2 & 4 \\
                  6 & 4 \\
                \end{array}
              \right].
\end{eqnarray*}
\end{example}

The \emph{maximum genus}, \emph{maximum non-orientable genus} and \emph{maximum Euler-genus}  of a graph $G$, denoted by $\gamma_{\max}(G),$ $\tilde{\gamma}_{\max}(G)$ and $\eps_{\max}(G),$ respectively, are given by $\gamma_{\max}(G)=\max \{i| \gamma_i(G)>0\},  \tilde{\gamma}_{\max}(G)=\max \{i| \tilde{\gamma}_i(G)>0\} \text{ and } \eps_{\max}(G)=\max \{i| \eps_i(G)>0\}$ respectively. One can see that $\eps_{\max}(G)=\max\{2\gamma_{\max}(G),\tilde{\gamma}_{\max}(G)\}.$ A \textit{cactus graph},  also called a \textit{cactus tree}, is a connected graph in which any  two graph  cycles have no vertex in common. Recall that a graph $G$ with orientable maximum genus $0$ if and only if $G$ is the cactus graph, and a graph $H$ of { maximum Euler-genus} $0$ if and only if $H$ is homeomorphic to  the path graph $P_n$ on $n$ vertices for $n>1$.

{The following two basic properties are followed by the definition of $H$-linear family  of graphs with spiders.
\begin{prop}\label{1}
For $1\leq j\leq k,$  $\dsum_{i=1}^{k}m_{ij}(1)$ are all the same constant $D$.
Moreover, for any $n\geq 2,$ we have  $\frac{P_n(1)}{P_{n-1}(1)}=D,$ where $P_n(x)=\Gamma_{G_n^\circ }(x)$
or $\cE_{G_n^\circ}(x).$
\end{prop}}

\begin{Rem}\label{rem:D}
In the proof of Theorem \ref{11-3a} below, we will see $D$ have the same meaning as that appeared in Section 2, so we use the same notation.
\end{Rem}

\begin{prop}\label{M:constant:G}
Suppose $M(x)$ is the production matrix of genus distribution (Euler-genus distribution) of { $\cG$}, then $M(x)$  is a constant if and only if the maximum genus (maximum Euler-genus) of $G_n$ equals { $0, \forall n\in \mN.$}
\end{prop}


\subsection{The limits for embedding  distributions of $H$-linear families of graphs with spiders}

A square matrix $A =(a_{i,j})_{i,j=1}^k$ is said to be non-negative   if
\begin{eqnarray*}
  a_{i,j}\geq 0,\quad \forall i,j=1,\cdots,k.
\end{eqnarray*}
  Let $A$ be a non-negative $k\times k$ matrix with maximal eigenvalue $r$ and suppose that $A$ has exactly $h$ eigenvalues of modulus $r.$ The number $h$ is called the {\textit{index}} of imprimitivity of $A.$ If $h=1,$ the matrix $A$ is said to be \emph{primitive}; otherwise, it is imprimitive.
  A square matrix $A =(a_{i,j})_{i,j=1}^k$ is said to a  stochastic matrix  if
 \begin{align*}
   \sum_{i=1}^ka_{ij}=1,  j=1,\cdots,k.
 \end{align*}

{
The following property of primitive stochastic matrix can be found in Proposition 9.2 in \cite{Beh}.
\begin{prop}\label{eige}
Every eigenvalue $\lambda$ of a stochastic matrix $A$  satisfies $|\lambda|\leq1.$ Furthermore, if the stochastic matrix $A$ is primitive, then all other eigenvalues of modulus are less than $1$, and algebraic multiplicity of $1$ is one.
\end{prop}}

\begin{thm}
\label{11-3a}
Consider the sequence of graphs $\cG^\circ=\{G_n^\circ\}_{n=1}^\infty$.
Let  $\{p_i(n), i=0,1\cdots,\}$
 be the  genus polynomial (Euler-genus polynomial) of  graph   $G_n^\circ$,
 $P_n(x)$ be the genus polynomial (Euler-genus polynomial) of $G_n^\circ$ and $M(x)$ be the production matrix for $\cG^\circ$.
If  the matrix   $M(1)$ is primitive,
then  the results of Theorem  \ref{11-3} hold.  Furthermore, if $M(x)$ is a constant,
  then the limit of the law for embedding distributions of $G_n^\circ$ is  a  discrete distribution.
\end{thm}
\begin{proof}

Suppose  that the characteristic polynomial of the production matrix $M(x)$ is
\begin{align*}
\label{55-1}
F(x,\lambda)&= \lambda^k-b_1(x)\lambda^{k-1}-b_2(x)\lambda^{k-2}\cdots-b_{k-1}(x)\lambda-b_k(x),
\end{align*}
{where  $b_j(x)$$(1\leq j \leq k)$ are polynomials with   integer coefficients}.
We also assume  $F(x,\lambda)=(\lambda-\lambda_1(x){)}^{m_1(x)}\cdots (\lambda-\lambda_r(x)^{m_r(x)}.$
Then, by the results in  \cite{CG19} (\cite{CG18}),  the sequence of genus polynomials (Euler-genus polynomial) of  $H$-linear family of graphs with spiders $G_n^\circ$ satisfy the following $k^{th}$-order linear recursion
\begin{eqnarray*}
\label{30-2}
  P_{n}(x)=b_1(x)P_{n-1}(x)+b_2(x)P_{n-2}(x)+\cdots+
  b_k(x)P_{n-k}(x).
\end{eqnarray*}

By Proposition \ref{1} and Proposition \ref{eige}, if $M(1)$ is primitive,
we have
\begin{eqnarray*}
D=\lambda_1(1)> |\lambda_2(1)|\geq |\lambda_3(1)|\geq \cdots |\lambda_r(1)|
\text{ and }  m_1(1)=1.
\end{eqnarray*}

 For any $n\in \mN,$  we denote the  embedding distribution of  graph   $G_n^\circ$  by $\{p_i(n),i=0,1,2,\cdots,\}$   and let    $X_n$ be  a random variable with distribution
\begin{eqnarray*}
\label{10-1}
\mP(X_n=i)=\frac{p_i(n)}{P_n(1)}, \quad i=0,1,\cdots,
\end{eqnarray*}
and
\begin{eqnarray}
\label{pp-1}
  e=\frac{\lambda_1'(1)}{D},\quad  v=\frac{-\big(\lambda_1'(1)\big)^2 +D\cdot  \lambda_1^{''}(1)
 +D\cdot\lambda_1'(1) }{D^2}.
\end{eqnarray}
So following the lines in the proof of Theorem \ref{11-3},
 we finish our proof. Furthermore if $M(x)$ is a constant, then all these  functions   $b_1(x),\cdots,b_k(x)$ are constant.
 Noting the  case \uppercase\expandafter{\romannumeral3}  in Theorem \ref{11-3}, this  theorem follows.
\end{proof}
The primitive of   the  matrix $M(1)$ is very important in our proof.
For this, we give the following example.

\begin{example}
\label{q-3}
Let
\begin{eqnarray*}
M(x)=
\left[
\begin{matrix}
  x+1&0\\
  0&2x
\end{matrix}
\right].
\end{eqnarray*}
$M(1)$ is imprimitive.
By calculation,  we obtain
\begin{eqnarray*}
\lambda_1(x)=\frac{3x+1+|x-1|}{2},\quad \lambda_2(x)=\frac{3x+1-|x-1|}{2}.
\end{eqnarray*}
In this case, we even don't have the   differentiability  of   $\lambda_1(x)$ at $x=1.$
\end{example}

\begin{Rem}
As pointed by Stahl \cite{SS91},  in all known cases, the production matrix $M(x)$ for the genus distributions of
any $H$-linear family of graphs is primitive at $x=1.$ Currently, we don't know whether this is true for general (or most) linear families of graphs.
\end{Rem}

In the rest of this subsection, we apply Theorem \ref{11-3a} to
path-like  and ladder-like sequences of graphs.

A vertex with degree $1$ is called a \textit{pendant vertex}, and the edge incident with that vertex is called a \textit{pendant edge.} If a pendant vertex $u$ of a graph $G$ is chosen to be a root, then the vertex $u$ is called a \textit{pendant root}. Suppose $(H,u,v)$ is a connected graph with two pendant roots $u,v.$ For $i=1,2,\ldots,n,$ let $(H_i,u_i,v_i)$ be a copy of $(H,u,v).$ By the way in subsection 3.1,  we construct a  $(H,u,v)$-linear family of graphs and a  $(H,u,v)$-linear family of graphs with spiders $(J_1,t_{1,i'})$ and $(\overline{J_1},{\overline{t_{1,i'}}}),$  where $(J_1,{t_{1,i'}})$ }and $(\overline{J_1},{\overline{t_{1,i'}}})$ be two connected graphs with roots $t_{1,i'}, \overline{t_{1,i'}}$ respectively.
For the $(H,u,v)$-linear family of graphs  with spiders or not, they have the same production matrix. Therefore,  we use the same notation $\{P_{n}^{H}\}_{n=1}^{\infty}$ to denote them.
For convenience, we call the graph $P_{n}^{H} $ \textit{path-like}.
 Figure  \ref{path}  demonstrate the graphs  $H$ and $P_{n}^{H} $, the shadow part of $H$ can be any connected graph.
\begin{figure}[h]
  \centering
  \includegraphics[width=0.75\textwidth]{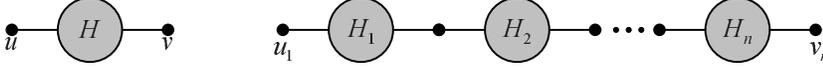}
  \caption{  Graph $H$ (left), and path-like graph $P^H_n$ (right)}
  \label{path}
\end{figure}
\begin{cor}
The genus distributions (Euler-genus distributions) of the path-like sequence  of graphs $ \{P_n^H\}_{n=1}^\infty$ with spiders  are
asymptotic normal distribution if the maximum genus  (maximum Euler-genus) of $(H,u,v)$ is greater than $0.$
\end{cor}
\begin{proof}Let $(H,u,v)$ be a graph with two pendant roots $u$ and $v.$  We introduce the following two \textit{partial genus distributions (partial Euler-genus distributions)} for $(H,u,v)$. Let $d_i(H)$ be  the number of embeddings of $(H,u,v)$ in the surface $O_i$ $(S_i)$ such that   $u$ and $v$  lie on  different face-boundary walks. In this case, we say that the embedding has type $d$.
 Similarly, let $s_i(H)$ be  the number of embeddings of $(H,u,v)$ in  $O_i$ $(S_i)$ such that  $u,v$ lie on the same face-boundary walk,
  and we call the embedding has type $s$. The two {partial genus polynomials (partial Euler-genus polynomials)} of $(H,u,v)$  are given by
$D_H(x)=\dsum_{i\geq0}d_i(H)x^i,$ and
$S_H(x)=\dsum_{i\geq0}s_i(H)x^i.$
Clearly,
\begin{eqnarray}
\label{zha-1}
  P_H(x)=D_H(x)+S_H(x),
\end{eqnarray}
where $P_H(x)$ is the genus polynomial (Euler-genus
polynomial) of $(H,u,v)$.
By face-tracing and Euler formula, we have the following claim.
\begin{description}\label{claim1}
  \item[Claim]  If the graph $(P_{n-1}^{H},u_1,v_{n-1})$ and $(H_n,u_n,v_n)$ embed on
 surfaces $O_i{ \ (S_i)} $ and $O_j{ \ (S_j)},$ respectively, then the graph $P_{n}^{H}$  embeds on $O_{i+j}{\ (S_{i+j})}$.
\end{description}

\noindent By using the claim above, we will build recurrence formulas for the partial genus polynomials of $P_{n}^{H}.$ There are four cases.

\begin{description}
  \item[Case 1] If both  the embeddings of $P_{n-1}^{H}$ and $(H_n,u_n,v_n)$  have  type $d$, then the embedding of $P_{n}^{H}$ has type $d.$ This case contributes to $D_{P_{n}^{H}}(x)$  the term $D_{P_{n-1}^H}(x)D_{H_n}(x).$
  \item[Case 2] If the embeddings of $P_{n-1}^{H}$ and $(H_n,u_n,v_n)$ have  type $d$ and $s,$ respectively, then the embedding of $P_{n}^{H}$ has type $d.$ This case contributes to $D_{P_{n}^{H}}(x)$  the term $D_{P_{n-1}^H}(x)S_{H_n}(x).$
  \item[Case 3] If the embedding of $P_{n-1}^{H}$ has  type $s$ and the embedding of  $(H_n,u_n,v_n)$ has type $d$, then the embedding of $P_{n}^{H}$ has type $d,$ this case contributes to $D_{P_{n}^{H}}(x)$  the term $S_{P_{n-1}^H}(x)D_{H_n}(x).$
  \item[Case 4] If both the embeddings of $P_{n-1}^{H}$ and $(H_n,u_n,v_n)$  have  type $s$, then the embedding of $P_{n}^{H}$ has type $s.$ This case contributes to $S_{P_{n}^{H}}(x)$  the term $S_{P_{n-1}^H}(x)S_{H_n}(x).$
\end{description}

The following linear recurrence system of equations summarizes the four cases above.
\begin{eqnarray}
D_{P_{n}^{H}}(x)&=&\left(D_{H_n}(x)+S_{H_n}(x)\right)D_{P_{n-1}^{H}}(x)+D_{H_n}(x)S_{P_{n-1}^{H}}(x),\label{orient:sum1} \\
S_{P_{n}^{H}}(x)&=&S_{H_n}(x)S_{P_{n-1}^{H}}(x).\label{orient:sum2}
\end{eqnarray}
Rewriting the equations above as
\begin{align*}
\left[
    \begin{array}{c}
    D_{P_{n}^{H}}(x) \\
    S_{P_{n}^{H}}(x) \\
    \end{array}
  \right] =\left[
   \begin{array}{cc}
     D_{H_n}(x)+S_{H_n}(x)  & D_{H_n}(x) \\
     0  & S_{H_n}(x) \\
   \end{array}\right]\left[
    \begin{array}{c}
    D_{P_{n-1}^{H}}(x) \\
    S_{P_{n-1}^{H}}(x) \\
    \end{array}
  \right].\end{align*}

\noindent Since $(H_n,u_n,v_n)$ be a copy of $(H,u,v),$ thus the production matrix $M(x)$ of the genus ({  Euler-genus}) distributions of { $\{P_{n}^{H}\}_{n=1}^\infty$}  is $$\left[
   \begin{array}{cc}
     D_{H}(x)+S_{H}(x)  & D_{H}(x) \\
     0  & S_{H}(x) \\
   \end{array}\right].$$

For simplicity of writing,  we use $D(x)$, $S(x)$ and $P(x)$ to denote $D_{H}(x)$, $S_{H}(x)$ and $P_H(x)$,   respectively.
Obviously, the eigenvalues of  matrix $M(x)$ are given by
\begin{eqnarray*}
  \lambda_1(x)=D(x)+S(x),\quad \lambda_2(x)=S(x).
\end{eqnarray*}
Since  the graph $H$ is connected,  $S(x)=0$ is impossible.
We make the following discussions on the $D(x)$

If  $D(x)=0,$ then  $H$ is the path graph $P_m$ on $m$ vertices ($m\geq 2$).
 In this case, the maximum genus  (maximum Euler-genus) of $(H,u,v)$ is equal  $0.$

 Now we consider the case   $D(x)\neq 0.$
Under this situation, $\lambda_1(1)>\lambda_2(1)$ and  the matrix $M(1)$ is primitive.
By direct calculation, we get
\begin{eqnarray*}
  e&=& \frac{D'(1)+S'(1)}{D(1)+S(1)}, \quad
  \\  v&=& \frac{-(D'(1)+S'(1))^2+\big(D(1)+S(1)\big)\big(D''(1)+S''(1)+D'(1)+S'(1)\big)}{(D(1)+S(1))^2}.
\end{eqnarray*}
We assume $P(x)=\sum_{m}c_mx^m.$
By  \textit{Cauchy-Schwarz inequality},
$$\left(\sum_{m}m c_m\right)^2\leq \sum_{m} m^2 c_m \cdot \sum_{m}c_m=
 \left(\sum_{m}( m^2-m) c_m+\sum_{m}mc_m\right) \cdot \sum_{m}c_m,
   $$
 which implies
  $P'(1)^2\leq (P''(1)+P'(1))\cdot P(1).$
  By this inequality and (\ref{zha-1}), one easily sees that $v\geq 0.$
Therefore, $v=0$ is equivalent  to the above  Cauchy-Schwarz inequality becomes equality, that is
\begin{eqnarray*}
\left(\sum_{m}m c_m\right)^2=\sum_{m} m^2 c_m \cdot \sum_{m}c_m.
\end{eqnarray*}
Since $P(x)\neq 0,$ we have $c_m>0$  for some $m\geq 0.$
The above equality is also
 equivalent  to that for some $x\in \mR$
 \begin{eqnarray*}
  m\sqrt{c_m}+x\sqrt{c_m}=0, \quad \forall m\geq 0.
 \end{eqnarray*}
 Therefore,  $v=0$ if and only if $\gamma_{max}(H)=\gamma_{min}(H)$ ($\eps_{\max}(H)= \eps_{\min}(H)$)
 .
 Noting that, a known fact in the topological  graph  theory says that $\gamma_{max}(H)=\gamma_{min}(H)$ ($\eps_{\max}(H)= \eps_{\min}(H)$)
 implies that $H$ is the cactus graph ($H$ is the path graph), we complete our proof.

\end{proof}

\begin{figure}[h]
  \centering
  \includegraphics[width=0.45\textwidth]{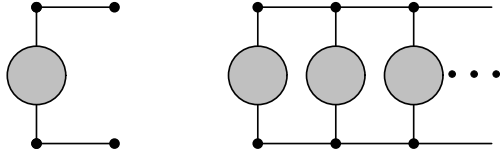}
  \caption{  Graph $H$ (left), and ladder-like graph $L^H_n$ (right)}
  \label{ladder}
\end{figure}

Given any graph $(H, u,v)$ whose root-vertices $u$ and $v$ are both 1-valent, as in Figure \ref{ladder},  we construct a \textit{ladder-like sequence of graphs}  $\{(L_{n}^H, u_{n}, v_{n})\}_{n=1}^\infty$ \cite{CGMT19}.  The shadow part of $H$ can be any connected graph, and the ladder-like sequences are a special case of $H$-linear family of graphs.
\begin{cor}The genus distributions of the ladder-like sequence  of graphs $ \{L_n^H\}_{n=1}^\infty$ are
asymptotic normal distribution.
\end{cor}
\begin{proof}
 By \cite{CGMT19}, the production matrix for genus distributions of the ladder-like sequence of graphs $ L_1^H,L_2^H,L_3^H,\cdots$  is
\small{
\begin{eqnarray*}
  M(x)=p(x)\left[\begin{matrix}
    4x&2x &0
    \\ 0& 0& 0
    \\ 0 &2x& 4x
  \end{matrix}
  \right]
  +q(x)\left[\begin{matrix}
    0&0 &0
    \\ 0& 2& 4
    \\ 4x &2x& 0
  \end{matrix}
  \right],
\end{eqnarray*}}
where
$p(x),q(x)\in \mZ(x)$ are  the partial genus polynomials for $H$. Also by \cite{CGMT19}, $q(x)=0$ is impossible.

Obviously, $D=4(p(1)+q(1))$
and  the eigenvalues of  matrix $M(1)$ are given by
\begin{eqnarray*}
  \lambda_1=4(p(1)+q(1)),~\lambda_2=4p(1)-2q(1),\lambda_3=0.
\end{eqnarray*}
By $p(1)\geq 0$ and $q(1)>0$, one easily  sees that
\begin{eqnarray*}
  \lambda_1>|\lambda_2|, \quad \lambda_1>|\lambda_3|
\end{eqnarray*}
and  the matrix    $M(1)$  is   primitive.

With the help of a computer, one arrives at that
\begin{eqnarray*}
&&e=\frac{3 \left(p'(1)+q'(1)\right)+3 p(1)+q(1)}{3 (p(1)+q(1))},
 \\  && v=\frac{4}{27(p(1)+q(1))}\Big[
2 q(1)^2-27 \big(p'(1)+q'(1)\big)^2
\\ && \quad +9 q(1) \Big(3p''(1)+3 q''(1)+3 q'(1)+7 p'(1)\Big)
\\ && \quad +p(1) \Big(27 p''(1)+27 p'(1)+27 q''(1)-9q'(1)+14 q(1)\Big)
\Big].
\end{eqnarray*}
In the rest of this corollary, we will prove  $v>0.$
Assume
\begin{eqnarray*}
  p(x)=\sum_{m}a_mx^m,\quad q(x)=\sum_mb_mx^m.
\end{eqnarray*}
Using \textit{Cauchy-Schwarz inequality} again, we  see  that
  $p'(1)^2\leq (p''(1)+p'(1))\cdot p(1)$
 and $q'(1)^2\leq (q''(1)+q'(1))\cdot q(1)$.
 Therefore, in order to prove   $v>0$, it is suffice to show   that
\begin{eqnarray*}
54p'(1)q'(1)+9p(1)q'(1)
<  2 q(1)^2 +9 q(1) \Big(3p''(1)+7 p'(1)\Big)
+p(1) \Big(27 q''(1)+14 q(1)\Big),
\end{eqnarray*}
which   is  equivalent  to
\begin{eqnarray*}
   \sum_{i,j}\Big[54ija_ib_j+9a_ijb_j\Big]
 <2 q(1)^2 + \sum_{i,j}\Big[27(i^2-i)+63i+27(j^2-j)+14\Big]a_ib_j.
\end{eqnarray*}
The above inequality is due to
\begin{eqnarray*}
 2 q(1)^2+ \sum_{i,j}\Big[27(i-j)^2+36(i-j)+14\Big]a_ib_j
 = 2 q(1)^2+ \sum_{i,j}\Big(27\big(i-j+\frac{2}{3}\big)^2+2\Big)a_ib_j>0.
\end{eqnarray*}
We complete the proof of $v>0.$

By Theorem \ref{11-3a},
the genus distributions of the ladder-like sequence  $ \{L_n^H\}_{n=1}^\infty$ are
asymptotic normal distribution  with mean $e\cdot n$ and variance $v\cdot n$.
\end{proof}

\subsection{{Limits for crosscap-number distributions of graphs}}

We  demonstrate the relationship between the limits of crosscap-number distributions and   Euler-genus distributions. 

 \begin{thm}
 \label{2-1}
    Let $a_n=\frac{\Gamma_{G_n^\circ}(1)}{\cE_{G_n^\circ}(1)},b_n=1-a_n.$
 If  $\lim\limits_{n\rightarrow \infty}a_n=0,$    we have
 \begin{eqnarray}
 \label{3-2}
  \lim_{n\rightarrow \infty}\sup_{x\in \mR}\Big|\sum_{0\leq i\leq x}\frac{\varepsilon_i(G_n^\circ)}{\cE_{G_n^\circ}(1)}-\sum_{0\leq i\leq x}\frac{\tilde{\gamma_i}(G_n^\circ)}{\tilde{\Gamma}_{G_n^\circ}(1)}\Big|=0,
 \end{eqnarray}
 which implies that
 the limits of crosscap-number distributions are the same as that of    Euler-genus distributions.
 \end{thm}
 \begin{proof}
 Since $a_n=\frac{\Gamma_{G_n^\circ}(1)}{\cE_{G_n^\circ}(1)},$ then $b_n=1-a_n=\frac{\tilde{\Gamma}_{G_n^\circ}(1)}{\cE_{G_n^\circ}(1)}.$
Let $\mU_n,\tilde{\mU}_n,\mW_n$ be three random variables with distributions given by the
genus, crosscap-number and  Euler-genus respectively, that is
\begin{eqnarray*}
  \mP(\mU_n=i)=
  \frac{\gamma_i(G_n^\circ)}{\Gamma_{G_n^\circ}(1)},
  ~~ \mP(\tilde{\mU}_n=i)=
  \frac{\tilde{\gamma}_i(G_n^\circ)}{\tilde{\Gamma}_{G_n^\circ}(1)},
  ~~\mP(\mW_n=i)=\frac{\varepsilon_i(G_n^\circ)}{\cE_{G_n^\circ}(1)}, ~i=0,1\cdots.
\end{eqnarray*}
Since
\begin{eqnarray*}
\label{1-3} \cE_G(x)=\Gamma_G(x^2)+\tilde{\Gamma}_G(x),
\end{eqnarray*}
 for any $i\in \mN,$ we have
\begin{eqnarray*}
&& \cE_{G_n^\circ}(1)\cdot  \mP(\mW_n=i)= \varepsilon_i(G_n^\circ)=\gamma_{\frac{i}{2}}(G_n^\circ)+\tilde{\gamma}_i(G_n^\circ)
\\ && =\mP(\mU_n=\frac{i}{2})\cdot \Gamma_{G_n^\circ}(1)+\mP(\tilde{\mU}_n=i)\cdot \tilde{\Gamma}_{G_n^\circ}(1),
\end{eqnarray*}
here $\gamma_{\frac{i}{2}}(G_n^\circ)$ is defined as $0$ when $i$ is an odd number.
Therefore, it holds that
\begin{eqnarray*}
  \mP(\mW_n=i)=a_n \mP(2\mU_n=i)+b_n\mP(\tilde{\mU}_n=i)
\end{eqnarray*}
 By this, for any $x\geq 0$,
 we have
 \begin{eqnarray*}
   \mP(\mW_n\leq x)-\mP(\tilde{\mU}_n\leq x)=
   a_n \mP(2\mU_n\leq x)+(b_n-1)\mP(\tilde{\mU}_n\leq x).
 \end{eqnarray*}
 Thus
 \begin{eqnarray*}
   \Big|\mP(\mW_n\leq x)-\mP(\tilde{\mU}_n\leq x)\Big|\leq
   a_n +(1-b_n).
 \end{eqnarray*}
 By the definitions of $\mW_n,\tilde{\mU}_n$,  the above inequality implies (\ref{3-2}).
\end{proof}

By  (\ref{3-2}), once the limits of   Euler-genus distributions is   obtained, the limits of crosscap-number distributions  is   also known.

With the same method as that in   Theorem \ref{2-1}, we obtain the following corollary.
\begin{cor}
\label{r-2}
  Let  $\{G_n\}_{n=1}^\infty$ be  any {sequence} of graphs,  which is not required to be  $H$-linear family of graphs with spiders. If $\lim\limits_{n\rightarrow \infty}\frac{\Gamma_{G_n}(1)}{\cE_{G_n}(1)}=\frac{1}{2^{\beta({G_n})}}=0$
  or $\lim\limits_{n\rightarrow  \infty}\beta(G_n)=\infty,$
  we have
 \begin{eqnarray*}
  \lim_{n\rightarrow \infty}\sup_{x\in\mR}\Big|\sum_{0\leq  i\leq x}\frac{\varepsilon_i(G_n)}{\cE_{G_n}(1)}-\sum_{0\leq  i\leq x}\frac{\tilde{\gamma_i}(G_n)}{\tilde{\Gamma}_{G_n}(1)}\Big|=0.
 \end{eqnarray*}
\end{cor}

\section{More examples and some research problems}

The graphs, which have explicit  formulas for their   embedding distributions,
 mainly are linear families of graphs with spiders, see \cite{CG18,CG19,CGMT19b,GKMT18,SS97} for details. There are  many linear families of  graphs with spiders which satisfy  the conditions of  Theorem \ref{11-3} or  Theorem   \ref{11-3a}.  However, we give a few examples to demonstrate the theorems, including the famous M\"{o}bius ladders, Ringel ladders and Circular ladders.

\begin{example}
Let $Y_n$ be the iterated-claw graph of Figure \ref{pp3}.
  \begin{figure}[h]
  \centering
  \includegraphics[width=0.45\textwidth]{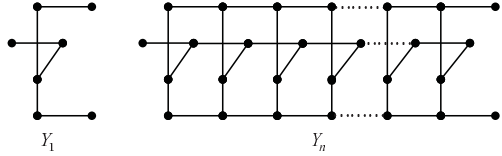}
  \caption{The claw $Y_1$ (left), and the iterated-claw graph $Y_n$ (right)}
  \label{pp3}
\end{figure}

\noindent {The genus polynomial  for the  iterated-claw graph $Y_n$ \cite{GMTW16-2} is given by} $$\Gamma_{Y_n}(x)=20 x\Gamma_{Y_{n-1}}(x)-8 (-3 x + 8 x^2)\Gamma_{Y_{n-2}}(x)-384 x^3 \Gamma_{Y_{n-3}}(x).$$
Since
\begin{eqnarray*}
  F(x,\lambda)&=&\lambda^3-20  x\lambda^2+8 (8 x^2-3 x) \lambda +384 x^3,
  \\ F(1,\lambda)&=& (\lambda-16)\big(\lambda-2 (1-\sqrt{7})\big)\big(\lambda-2 (\sqrt{7}+1)\big),
\end{eqnarray*}
the conditions of Theorem \ref{11-3}  hold for $P_n(x)=\Gamma_{Y_n}(x)$ and $\lambda_1(1)=D=16$.
 By the derivative rule of implicit function and with the help of Maple,   one sees
 \begin{eqnarray*}
   e=\frac{6}{7},\quad v=\frac{8}{147}>0.
 \end{eqnarray*}
 For the constants  $e,v$ given above,
we have
    \begin{eqnarray*}
   \lim_{n\rightarrow \infty } \sup_{x\in \mR }\left|\frac{1}{\Gamma_{Y_n}(1) }\sum_{0\leq  i\leq  x \sqrt{ v\cdot n}+e\cdot n}\gamma_{i}(Y_n)-\int_{-\infty}^x \frac{1}{\sqrt{2\pi}}e^{-\frac{1}{2}u^2}\dif u\right|=0.
  \end{eqnarray*}

 By \cite{CG18}, the Euler-genus polynomials of $Y_n$ satisfy the following three-order recurrence relation
 \begin{eqnarray*}
   \cE_{Y_n}(x)&=& 2(3x+28x^2) \cE_{Y_{n-1}}(x)-16(-3x^2-12x^3+4x^4) \cE_{Y_{n-2}}(x)-3072x^6 \cE_{Y_{n-3}}(x).
 \end{eqnarray*}
 Set  $P_n(x)=\cE_{Y_n}(x)$ and
 \begin{eqnarray*}
   F(x,\lambda)=\lambda^3- 2(3x+28x^2) \lambda^2+16(-3x^2-12x^3+4x^4)\lambda+3072x^6.
 \end{eqnarray*}
 Since $F(1,\lambda)=(\lambda-64) (\lambda-6) (\lambda+8)$, the conditions  of Theorem \ref{11-3}   hold and  $D=\lambda_1(1)=64$.
 By the derivative rule of implicit function and with the help of Maple,   one sees
 \begin{eqnarray*}
 \label{z-1}
  e=\frac{160}{87},v=\frac{269092}{1975509}>0.
 \end{eqnarray*}
For the constants  $e,v$ given above,
the Euler-genus distributions of $Y_n$ are asymptotic normal distribution  with mean $e\cdot n$ and variance $v\cdot n$. By Theorem \ref{2-1}, the crosscap-number distributions of $Y_n$ are also
asymptotic normal distribution with mean $e\cdot n$ and variance $v\cdot n$.

 \end{example}

\begin{example}Let $G_n=P_n\Box P_3$ be the \mdef{grid graph} \cite{KPG12}. {The genus polynomials for the grid graph $G_n$  are given by  the recursion}
\begin{eqnarray}
\Gamma_{G_n}(x)&=&(1 + 30 x)\Gamma_{G_{n-1}}(x)-42 (-x + 4 x^2)\Gamma_{G_{n-2}}(x)\notag\\
&&-72 (x^2 + 14 x^3)\Gamma_{G_{n-3}}(x)+1728 x^4\Gamma_{G_{n-4}}(x).\notag
\end{eqnarray}
With the help of Maple,
the conditions of Theorem \ref{11-3} hold for $P_n(x)=\Gamma_{G_n}(x)$ and
\begin{eqnarray*}
  D=\lambda_1(1)=24, ~e=\frac{34}{41},~v=\frac{4816}{68921}>0.
\end{eqnarray*}

{The Euler-genus polynomials for the grid graphs $G_n$ \cite{CG18} satisfy the recursion }
\begin{eqnarray}\label{Euler:grid}
\quad\qquad \cE_{G_n}(x) &=& (1 + 11 x + 84 x^2)\cE_{G_{n-1}}(x)+ 12 x^2 (7 + 30 x - 28 x^2)\cE_{G_{n-2}}(x) \notag\\
&& -288 x^4 (1 + 4 x + 32 x^2)\cE_{G_{n-3}}(x) + 27648 x^8\cE_{G_{n-4}}(x).  \notag
\end{eqnarray}
With the help of Maple,
the conditions of Theorem \ref{11-3} hold for $P_n(x)=\cE_{G_n}(x)$   and
\begin{eqnarray*}
  D=\lambda_1(1)=96, ~e=\frac{5488}{3037},~v=\frac{4819233780}{28011371653}>0.
\end{eqnarray*}

By the discussions above, the embedding distributions of the grid graph  $G_n$ are asymptotic normal distributions.
\end{example}

\begin{example}\label{ladders}{Suppose that $n$ is a positive integer. Let $C_n$ be the cycle graph on $n$ vertices.  The Ringel ladder $R_n$ is obtained  by adding
an edge joining the two vertices of the leftmost edge and rightmost edge of the ladder graph $L_n$. A \mdef{circular ladder} $CL_n$ is the graphical Cartesian product $CL_n=C_n\Box P_2.$ The \mdef{M\"{o}bius ladder} $ML_n$ is formed from an $2n$-cycle by adding edges connecting opposite pairs of vertices in the cycle.  i.e., the M\"obius ladder can be described as a circular ladder with a half-twist. It is known in \cite{CG19} that the  Ringel ladder, circular ladder, and M\"{o}bius ladder are ring-like  families of graphs. Let $H_n$ be the  Ringel ladder $R_n$, circular ladder $CL_n$, or M\"{o}bius ladder $ML_n.$}


{By Theorem 3.1 in \cite{CG19}, the genus polynomials for  $H_n$ \cite{CG19} satisfy the recursion}
\begin{eqnarray}\label{Genus:circular:re}
\Gamma_{H_n}(x) &=& 4\Gamma_{H_{n-1}}(x) + (-5+20x)\Gamma_{H_{n-2}}(x)
	+ (-56x+2)\Gamma_{H_{n-3}}(x)\notag\\
&&-4(32x-11)x\Gamma_{H_{n-4}}(x) + 8(28x-1)x\Gamma_{H_{n-5}}(x)\notag\\
&&+32(8x-3)x^2\Gamma_{H_{n-6}}(x) - 256x^3\Gamma_{H_{n-7}}(x).\notag
\end{eqnarray}
Since
\begin{eqnarray*}
  F(x,\lambda)&=&\lambda^7- 4\lambda^6 - (-5+20x)\lambda^5
	- (-56x+2)\lambda^4
\\ &&
+4(32x-11)x\lambda^3 - 8(28x-1)x\lambda^2
-32(8x-3)x^2\lambda + 256x^3
\\ &=& (\lambda-1)(\lambda+2\sqrt{x})(\lambda-2\sqrt{x})
(\lambda-1+\sqrt{1+8x})(\lambda-1-\sqrt{1+8x})
\\ &&  *\big(\lambda-\frac{1}{2} (1-\sqrt{32 x+1})\big)\big(\lambda-\frac{1}{2} (1+\sqrt{32 x+1})\big),
\end{eqnarray*}
the conditions of Theorem  \ref{11-3} hold  for $P_n(x)=\Gamma_{H_n}(x)$ and $\lambda_1(x)=\sqrt{8 x+1}+1, D=\lambda_1(1)=4$.
With the help of Maple, one easily sees
\begin{eqnarray*}
  e=\frac{1}{3}, ~v=\frac{2}{27}>0.
\end{eqnarray*}
 By Theorem  \ref{11-3}, for the constants  $e,v$ given above ,
we have
    \begin{eqnarray*}
   \lim_{n\rightarrow \infty } \sup_{x\in\mR}\left|\frac{1}{\Gamma_{H_n}(1) }\sum_{0\leq  i\leq  x \sqrt{ v\cdot n}+e\cdot n}\gamma_{i}(H_n)-\int_{-\infty}^x \frac{1}{\sqrt{2\pi}}e^{-\frac{1}{2}u^2}\dif u\right|=0.
  \end{eqnarray*}

{The Euler-genus polynomials  for $H_n$ \cite{CG18} are  given by the recursion}
\begin{eqnarray} \label{Euler:Mobius:re}
\cE_{H_n}(x)&=&(12x+4)\cE_{H_{n-1}}(x)+(-12x^2-34x-5)\cE_{H_{n-2}}(x)\notag\\
&&+(-240x^3-20x^2+26x+2)\cE_{H_{n-3}}(x)\notag\\
&&+4(80x^3+128x^2+14x-1)x\cE_{H_{n-4}}(x)\notag\\
&&+16(112x^3+8x^2-14x-1)x^2{\cE_{H_{n-5}}(x)}\notag\\
&&-128(8x^2+18x+3)x^4{\cE_{H_{n-6}}(x)}\notag\\
&&-2048(2x+1)x^6\cE_{H_{n-7}}(x).\notag
  \end{eqnarray}
Since
\begin{eqnarray*}
  F(x,\lambda)&=&\lambda^7 -(12x+4)\lambda^6-(-12x^2-34x-5)\lambda^5
-(-240x^3-20x^2+26x+2)\lambda^4
\\ &&
-4(80x^3+128x^2+14x-1)x\lambda^3
-16(112x^3+8x^2-14x-1)x^2\lambda^2
\\ &&+128(8x^2+18x+3)x^4\lambda
+2048(2x+1)x^6,
\\ F(1,\lambda)&=& (\lambda+2)^2(\lambda-3)(\lambda-4)(\lambda-8)
\big(\lambda-\frac{1}{2} (5-\sqrt{89})\big)
\big(\lambda-\frac{1}{2} (5+\sqrt{89})\big),
\end{eqnarray*}
the conditions of Theorem  \ref{11-3} hold  for $P_n(x)=\cE_{H_n}(x)$  and $D=\lambda_1(1)=8.$
With the help of Maple, we have
\begin{eqnarray*}
\lambda_1(x)=\sqrt{20 x^2+4 x+1}+2 x+1,~e=\frac{4}{5}, ~v=\frac{22}{125}>0.
\end{eqnarray*}
By Theorem \ref{11-3}, for the constants  $e,v$ given above,
 the Euler-genus  distributions   of $H_n$  is    asymptotic normal distribution
 with mean $e\cdot n$ and variance $v\cdot n.$

\end{example}

\skipit{
     \begin{figure}[h]
  \centering
  \includegraphics[width=0.50\textwidth]{ringel1}
  \caption{The Ringel Graph $R_n$}
  \label{pp1}
\end{figure}

\begin{example}
\label{1-7}
  By \cite{CG19}, the genus polynomials of Ringel ladder graph $R_n$, as shown in Figure \ref{pp1}, satisfy the following fourth-order recurrence relation
  \begin{eqnarray*}
    \Gamma_{R_n}(x)&=& 3\Gamma_{R_{n-1}}(x)+(16x-2)\Gamma_{R_{n-2}}(x)
    -24 x \Gamma_{R_{n-3}}(x)-64x^2\Gamma_{R_{n-4}}(x)
  \end{eqnarray*}
Let $P_n(x)=\Gamma_{R_n}(x).$ Obviously,
  \begin{eqnarray*}
    F(x,\lambda)&=&  \lambda^4- 3\lambda^3-(16x-2)\lambda^2
    +24 x \lambda+64x^2,
    \\
    F(1,\lambda)& =& (\lambda-4)(\lambda+2)(\lambda-\frac{1}{2} (1-\sqrt{33}))(\lambda-\frac{1}{2} (1+\sqrt{33})).
  \end{eqnarray*}
 Therefore, the conditions of Theorem \ref{11-3} hold
 and
   \begin{eqnarray*}
   \begin{split}
  &  \lambda_1(x)=\sqrt{8 x+1}+1,  ~D=\lambda_1(1)=4, ~e=\frac{\lambda_1'(1)}{D}=\frac{1}{3},
   \\
&    v=\frac{-\big(\lambda_1'(1)\big)^2 +D \lambda_1^{''}(1)
 +D\lambda_1'(1) }{D^2}=\frac{2}{27}.
 \end{split}
   \end{eqnarray*}
By  Theorem  \ref{11-3}, for   the constants  $e,v$ given above,
we have
    \begin{eqnarray*}
   \lim_{n\rightarrow \infty } \frac{1}{\Gamma_{R_n}(1) }\sum_{\tiny{k\in \mN,k\in (a\sqrt{v\cdot n}+e\cdot n,~b\sqrt{ v\cdot n}+e\cdot n]}}\gamma_{k}(R_n)=\int_{a}^b \frac{1}{\sqrt{2\pi}}e^{-\frac{1}{2}u^2}\dif u.
  \end{eqnarray*}
  Especially, we have
      \begin{eqnarray*}
   \lim_{n\rightarrow \infty } \frac{1}{\Gamma_{R_n}(1) }\sum_{\tiny{k\in \mN,k\in (-2\sqrt{v\cdot n}+e\cdot n,~2\sqrt{ v\cdot n}+e\cdot n]}}\gamma_{k}(R_n)=\int_{-2}^2 \frac{1}{\sqrt{2\pi}}e^{-\frac{1}{2}u^2}\dif u=0.9544.
  \end{eqnarray*}

The Euler-genus polynomials for the Ringel ladder graph  $R_n$ satisfy the following fourth-order recurrence relation
 \begin{eqnarray*}
   \cE_{R_n}(x)&=& (8x+3)\cE_{R_{n-1}}(x)+(16x^2-12x-2)\cE_{R_{n-2}}(x)
\\  &&-16(8x+3)x^2\cE_{R_{n-3}}(x)-256x^4 \cE_{R_{n-4}}(x).
 \end{eqnarray*}
Let    $P_n(x)=\cE_{R_n}(x)$. We see
 \begin{eqnarray*}
 F(1,\lambda)& =
 &  \lambda^4- 11 \lambda^3-2\lambda^2
+176\lambda+256
\\ & = & (\lambda-8)
(\lambda+2)
\big(\lambda- \frac{1}{2}(5-\sqrt{89})\big)\big(\lambda-\frac{1}{2}(5+ \sqrt{89})\big).
 \end{eqnarray*}
Thus,  the conditions of  Corollary  \ref{11-3a} hold.
By calculation, we obtain
   \begin{align*}
   &D=\lambda_1(1)=8,  ~e=\frac{4}{5},  ~v=\frac{22}{125}.
   \end{align*}
   By Corollary  \ref{11-3a} and Theorem \ref{2-1}, for the constants  $e,v$ given above,
we have
    \begin{eqnarray*}
   \lim_{n\rightarrow \infty } \frac{1}{\cE_{R_n}(1) }\sum_{\tiny{i\in \mN,i\in (-2\sqrt{v\cdot n}+e\cdot n,~2\sqrt{ v\cdot n}+e\cdot n]}}\varepsilon_{i}(R_n)=\int_{-2}^2 \frac{1}{\sqrt{2\pi}}e^{-\frac{1}{2}u^2}\dif u=0.9544
  \end{eqnarray*}
and
    \begin{eqnarray*}
   \lim_{n\rightarrow \infty } \frac{1}{\tilde{\Gamma}_{R_n}(1)  }\sum_{\tiny{j\in \mN,j\in (-2\sqrt{v\cdot n}+e\cdot n,~2\sqrt{ v\cdot n}+e\cdot n]}}\tilde{\gamma}_{j}(R_n)=\int_{-2}^2 \frac{1}{\sqrt{2\pi}}e^{-\frac{1}{2}u^2}\dif u=0.9544.
  \end{eqnarray*}
\end{example}
}

\subsection{Some researches problems}

In the end of this paper,  we demonstrate some  research problems.

 \begin{question}
In our paper, the limits of the  embedding  distributions   for graphs  are normal distributions or  some  discrete distributions. Can we prove that the limit of the   embedding  distributions   for any $H$-linear family of graphs  is a normal distribution or a  discrete distribution. Furthermore, if the maximum genus (Euler-genus) $\eps_{\max}(G)$ of $H$ is great than $0,$ do we have the  limit for the genus  distributions (Euler-genus distribution) for any $H$-linear family of graphs  is a normal distribution?
 \end{question}

 \begin{question}\label{P2}
  Suppose that $\{\widetilde{G}_n\}_{n=1}^\infty$  is a family of graphs with $\beta(\widetilde{G}_n)\rightarrow \infty$ (orientable maximum genus $\gamma_M(\widetilde{G}_n)\rightarrow \infty$). Are the crosscap-number distributions (genus distributions) of $\widetilde{G}_n$  asymptotic normal. \end{question}

A \textit{bouquet of circles} $B_n$ is define as a graph with one vertex and $n$ edges. A \textit{dipole}  $D_n$ is a graph with two vertices joining by $n$ multiple edges. In \cite{GRT89}, Gross, Robbins, and Tucker obtained a second-order recursion for the genus distributions of $B_n.$ The genus distributions of $D_n$ were obtained by Rieper in \cite{Rie87}, and independently by Kwak and Lee in \cite{KL95}. The \mdef{wheel} $W_n$ is a graph formed by connecting a single vertex to each of the vertices of a $n$-cycle. The genus distribution of $W_n$ was obtained in \cite{CGM18} by Chen, Gross and Mansour. The following special case of question \ref{P2} is may not hard.

 \begin{question}  Are the embedding distributions of $B_n$, $D_n$ and $W_n$ asymptotic normal?
 \end{question}

The following question is for the complete graph $K_n$ on $n$ vertices and complete bipartite graph $K_{m,n}.$

 \begin{question}  Are the embedding distributions of $K_n$, and $K_{m,n}$ asymptotic normal?
 \end{question}

Another question is the following.
 \begin{question}
   Is the embedding distribution of a random graph on $n$ vertices  asymptotic normal?
 \end{question}

 \noindent$\bf{Acknowledgement}$
Yichao Chen was supported by NNSFC under Grant No.11471106. Xuhui Peng was supported  by Hunan Provincial Natural Science Foundation of  China under Grant 2019JJ50377.


\end{document}